\documentclass[a4paper, 10pt]{article}

\usepackage[T1]{fontenc}
\usepackage[utf8]{inputenc}

\usepackage{amssymb}
\usepackage{amsmath}
\usepackage{amsfonts}
\usepackage{amsthm}
\usepackage{mathtools}

\usepackage{fixltx2e} 
\usepackage{isomath} 

\usepackage{float}
\usepackage{caption}

\usepackage{pgfplots}

\newcommand{\E}{\mathbb{E}}

\DeclareMathOperator{\Cov}{Cov} 
\DeclareMathOperator{\sgn}{sgn} 
\DeclareMathOperator{\Var}{Var} 
\DeclareMathOperator{\unif}{unif}
\DeclareMathOperator{\bin}{Bin}

\theoremstyle{plain}%
\newtheorem{theorem}{Theorem}[section]
\newtheorem{lemma}[theorem]{Lemma}
\newtheorem{proposition}[theorem]{Proposition}
 
\newtheorem{claim}[theorem]{Claim} 

\theoremstyle{definition}
\newtheorem{definition}[theorem]{Definition}

\theoremstyle{remark}
\newtheorem{remark}[theorem]{Remark}

\title{Denseness of volatile and nonvolatile sequences of functions}
\author{Malin Palö Forsström\footnote{Department of Mathematics, Chalmers University of Technology and Gothenburg University, Sweden.
E-mail:  palo@chalmers.se.}}
\date{\today}

\begin{document}

\maketitle

\begin{abstract}
In a recent paper by Jonasson and Steif, definitions to describe the volatility of sequences of Boolean functions, \( f_n \colon \{ -1,1 \}^n \to \{ -1,1 \} \) were introduced. We continue their study of how these definitions relate to noise stability and noise sensitivity. Our main results are that the set of volatile sequences of Boolean functions  is a natural way "dense" in the set of all sequences of Boolean functions, and that the set of non-volatile Boolean sequences is not "dense" in the set of noise stable sequences of Boolean functions.  
\end{abstract}

\section{Introduction}

This paper will be concerned with the  volatility of   sequences of Boolean functions as defined in~\cite{js2015}, and in particular with the relation between the set of volatile sequences of Boolean functions and the sets of noise stable and noise sensitive sequences of Boolean functions respectively (see e.g.~\cite{gs2015}). All of  these definitions can be said to describe aspects of the behaviour of the value of a Boolean function when its input evolves according to a simple  Markov chain \( (X_t) \).

The Markov chain \( (X_t^{(n)})_{t \geq 0} \), \(X_t^{(n)} = (X_t^{(n)}(1), \ldots, X_t^{(n)}(n)) \), with which we will be concerned will have \( \{ -1,1 \}^n \) as its state space. We define the Markov chain by letting each coordinate update independently according to an exponential clock with rate one, setting the value at an updating coordinate to \( 1 \) with probability \( p_n \) and to \( -1 \) with probability \( 1-p_n \). Clearly, the stationary measure \( \pi_{p_n} \) for this process will be \( \{ 1-p_n , p_n \}^n \), and whenever nothing else is  written explicitly we will pick \( X_0^{(n)} \) according to this measure. To stress the dependence of \( (p_n) \) we sometimes add \( p_n \) as a subscript to \( P \), and write \( P_{p_n} \), \( \Cov_{p_n} \) etc.. 

If we compare the value of the process \( (X_t^{(n)}) \) at a coordinate \( i \) for \( t = 0 \) and \( t = \varepsilon \), we get
\begin{equation*}
\begin{cases}
P_{p_n}\left[X^{(n)}_\varepsilon(i) = 1 \mid X_0^{(n)}(i) = 1\right] = e^{-\varepsilon} + (1 - e^{-\varepsilon})p_n \vspace{1ex}\cr
P_{p_n}\left[X^{(n)}_\varepsilon(i) = -1 \mid X_0^{(n)}(i) = 1\right] =  (1 - e^{-\varepsilon})(1-p_n) \vspace{1ex}\cr
P_{p_n}\left[X^{(n)}_\varepsilon(i) = -1 \mid X_0^{(n)}(i) = -1\right] = e^{-\varepsilon} + (1 - e^{-\varepsilon})(1-p_n) \vspace{1ex}\cr
P_{p_n}\left[X^{(n)}_\varepsilon(i) = 1 \mid X_0^{(n)}(i) = -1\right] = (1 - e^{-\varepsilon})p_n .
\end{cases}
\end{equation*}
Consequently, \( X_\varepsilon^{(n)} \) can be thought of as being obtained by resampling each coordinate according to \( \{ 1-p_n, p_n \} \) with probability \( 1 - e^{-\varepsilon} \).

Whenever the dependency on \( n \) is clear, we will drop \( n \) in the superscript of \( (X_t^{(n)}) \).

The concept of noise sensitivity of sequences of Boolean functions was first defined in~\cite{schramm2000} as a measure of to what extent  knowledge about \( f_n(X_0) \) would help to predict \( f_n(X_\varepsilon) \). Our definition is the same as the definition used is e.g.~\cite{gs2015}, and is equivalent to what is called being \emph{asymptotically noise sensitivity} in~\cite{schramm2000}.
\begin{definition}
A sequence of Boolean functions \( f_n \colon \{ -1,1 \}^n \to \{ -1,1 \} \)  is said to be {\emph{noise sensitive}} with respect to \( (p_n) \)  if for all \( \varepsilon > 0 \) 
\begin{equation*}
\lim_{n \to \infty}  \Cov_{p_n}\left[f_n(X_0), f_n(X_\varepsilon) \right] = 0.
\end{equation*} 
\end{definition}

In the same paper, the authors also introduced the concept of noise stability, which captures a possible opposite behaviour.
\begin{definition}
A sequence of Boolean functions \( f_n \colon \{ -1,1 \}^n \to \{ -1,1 \} \) is said to be \emph{noise stable} with respect to \( (p_n) \)  if 
\begin{equation}\label{equation: definition of noise stability}
\lim_{\varepsilon \to 0} \sup_n P_{p_n}\left[f_n(X_\varepsilon) \not = f_n(X_0)\right] = 0.
\end{equation}
\end{definition}
Note that as \( \{ -1,1 \}^n \) is finite for every \( n  \in \mathbb{N} \) ,~\eqref{equation: definition of noise stability}~is equivalent to 
\[
\lim_{\varepsilon \to 0} \limsup_{n \to \infty} P_{p_n}\left[f_n(X_\varepsilon) \not = f_n(X_0)\right] = 0.
\]
When using these definitions, one generally assumes that the sequence \( (f_n) \) of Boolean  functions is \emph{nondegenerate}, meaning that
\[
-1 < \liminf_{n \to \infty}  \E [f_n(X_0)] \leq \limsup_{n \to \infty} \E [f_n(X_0)] <1.
\]
It is easy to show that if \( (f_n) \) is not nondegenerate, then it is both noise sensitive and noise stable.

In~\cite{js2015}, another measure of the stability of a sequence of Boolean functions was introduced. 
One motivation was that the two definitions above, although giving information about  \( f_n(X_t) \)  at two distinct times \( t = 0 \) and \( t = \varepsilon \), gives no information about  \( f_n(X_t) \) for intermediate times \(t \).
\begin{definition}\label{def: volatility}
A sequence of Boolean functions \( f_n \colon \{ -1,1 \}^n \to \{ -1,1 \} \)  is said to be \emph{volatile} with respect to \( (p_n) \) if for all \( \delta > 0 \), 
\[
\lim_{n \to \infty}  P_{p_n}\left[\tau_{\partial f_n} > \delta\right] = 0,
\]
where   \( \tau_{\partial f_n} \) is the hitting time of the set \( \{x \colon f_n(x) \not = f_n(X_0) \} \).
\end{definition}
It is shown in~\cite{js2015} that the property above is independent of \( \delta \).

When introducing a new definition, it is natural to ask how it relates to  earlier and by now well established definitions. Some results in this direction were given in~\cite{js2015}, such as that all nondegenerate, noise sensitive sequences of Boolean functions are volatile.
Throughout this paper we will say that two sequences  of Boolean of functions, \( f_n,g_n \colon \{ -1,1 \}^n \to \{-1,1 \} \), are \( o(1) \)-close with respect to \( (p_n) \) if 
\[
\lim_{n \to \infty} P_{p_n}\left[f_n(X_0) \not = g_n(X_0^{(n)})\right] = 0.
\]
Using this terminology, our first result complements the   results in~\cite{js2015} by stating that any sequence of Boolean functions is \( o(1)\)-close to a volatile sequence of Boolean functions, and hence, in some sense the set of volatile Boolean functions is dense in the set of all Boolean functions. 
\begin{theorem}\label{theorem: striped example}
If 
\begin{equation}\label{eq: natural condition}
\lim_{n \to \infty} np_n(1-p_n) = \infty
\end{equation}
then any sequence of Boolean functions \( f_n \colon \{ -1,1 \}^n \to \{ -1,1 \} \) is \( o(1) \)-close to a volatile sequence of Boolean functions with respect to \( (p_n) \).
\end{theorem}

Note that an analogous result does not hold if we replace the word \emph{volatile} above with the word \emph{noise sensitive}, as all sequences that are \( o(1) \)-close to a noise stable sequence of Boolean functions will  be noise stable, and hence not noise sensitve.

The condition that \( \lim_{n \to \infty} np_n(1-p_n) = \infty \) might seem odd, but is in fact quite natural. To see this, note that given the definition of \( ( X_t^{(n)} ) \), the expected number of coordinates  whose value has changed at least once between  time \( 0 \) and time \( \varepsilon \) is given by
\begin{equation*}
\begin{split}
& n \cdot \left( p_n \left( 1 - e^{-\varepsilon (1-p_n)} \right) + (1-p_n) \left( 1 - e^{-\varepsilon p_n} \right) \right)
\\&\qquad\leq  n \cdot \left( p_n \cdot 2\varepsilon (1-p_n)e^{-\varepsilon (1-p_n)}+ (1-p_n) \cdot 2 \varepsilon p_n e^{-\varepsilon p_n}  \right)
\\&\qquad=  2\varepsilon \cdot n p_n (1-p_n) \cdot \left( e^{-\varepsilon (1-p_n)}+  e^{-\varepsilon p_n}  \right)
\\&\qquad=  2\varepsilon \cdot n p_n (1-p_n) 
\end{split}
\end{equation*}
If \( np_n(1-p_n) \) is bounded from above, then this expression will tend to \( 0 \) as \( \varepsilon \to 0 \), uniformly in \( n \), hence any sequence of Boolean functions will be noise stable, and no sequence of Boolean functions will be volatile.

Throughout this paper, we will say that a set \( \mathcal F \) of sequences of Boolean functions is \emph{dense} in a set \( \mathcal{G} \) of sequences of Boolean functions if for each \( (g_n) \in \mathcal{G} \) there is \( (f_n) \in \mathcal{F} \) such that \( (f_n) \) and \( (g_n) \) is \( o(1)\)-close. Note that we have not assumed that \( \mathcal{F}  \subseteq \mathcal{G} \). With this terminology, we think of Theorem~\ref{theorem: striped example} as saying that the set of all volatile sequences of Boolean functions is dense in the set of all sequences of Boolean functions.

One consequence of Theorem~\ref{theorem: striped example} is that  for any sequence of Boolean functions, by just disturbing the functions in the sequence a little, we can make the sequence volatile. A natural question to ask is then if the reverse also holds, that is whether we can disturb any volatile function by a little to make it non-volatile, or in other words, if the set of non-volatile sequences of Boolean functions is dense in the set of all Boolean functions. However, as we already know that all nondegenerate, noise sensitive sequences of Boolean functions are volatile, and any nondegenerate sequence of Boolean functions which is \( o(1) \)-close to a noise sensitive sequence of Boolean functions must also be noise sensitive, such a converse cannot exist. On the other hand, we could still  ask if there could be a converse for all sequences of Boolean functions that are not noise sensitive, or even weaker, for Boolean functions that are noise stable. The main purpose of the next theorem is to show that no such converse can exist.

\begin{theorem}\label{theorem: noise stable and volatile}

When \( p_n = 0.5 \) there is a nondegenerate and noise stable sequence of Boolean functions \( f_n\colon \{-1,1\}^n \to \{ -1,1 \} \) that is not \( o(1)\)-close to any non-volatile sequence of Boolean functions. Moreover, \( (f_n) \) can be chosen to be either
\begin{enumerate}
\item symmetric and monotone, or
\item invariant under permutations of the coordinates.
\end{enumerate}

\end{theorem}

In terms of denseness, this theorem says that for \( p_n = 0.5 \), the set of non-volatile sequences of Boolean functions is not dense in the set of noise stable sequences of Boolean functions. Even stronger, it says that  for \( p_n = 0.5 \), the set of non-volatile sequences of Boolean functions is not dense neither in  the set of monotone noise stable sequences of Boolean functions nor in the set of noise stable Boolean functions that are invariant under permutations of the coordinates.

\begin{remark}
Similar constructions as the constructions used in the proof of this theorem do work for more general sequences \(( p_n) \) as well, at least as long as \( 0 < \liminf_{n \to\infty} p_n \leq \limsup_{n \to \infty} p_n < 1 \). A proof of this will however not be included in this paper.
\end{remark}

\begin{remark}

In comparison to Theorem~\ref{theorem: noise stable and volatile}, it is quite easy to see that when \( {p_n = 0.5} \), any sequence of Boolean functions \( (f_n) \) is arbitrarily close to a non-volatile sequence \( (f_n')\) of Boolean functions. Given \( (f_n) \) and \( k \in \mathbb{Z}_+\), define
\begin{equation*}
f_n'(x) = \begin{cases} 1 &\textnormal{if } x(1) = \ldots = x(k) = -1 \cr f_n(x) &\textnormal{else}.\end{cases}
\end{equation*}
Then
\begin{equation*}
P\left[f_n(X_0) \not = f_n'(X_0)\right]\leq P\left[X_0(1) = \ldots = X_0(k) = -1\right] = 2^{-k}.
\end{equation*}
This shows that \( (f_n) \) and \( (f_n') \) can be made arbitrarily close by choosing \( k \) large. However, for any fixed value of \( k \),
\begin{equation*}
	\begin{split}
		P\left[\tau_{\partial f_n'}>\delta\right] &\geq P\left[X_t^{(n)}(1) = \ldots = X_t^{(n)}(k)  = -1 \text{ for all } t \in [0,\delta]  \right]
		\\& \geq 2^{-k}  e^{-\delta k}
	\end{split}
\end{equation*}
and hence \( (f_n' ) \) will not be volatile.

\end{remark}

\begin{remark}
The example in the previous remark can easily be extended to more general sequences \( (p_n) \). If \( p_n \to 0 \) but we still have that \( \lim_{n \to \infty} np_n(1-p_n) = \infty \), then with a bit more careful analysis the proof above works if we let \( k \) depend on \( n \) and set \( k(n) = Cp_n^{-1} \) for some constant \( C>0 \). If instead \( p_n \to 1 \), a similar construction works, but we will also have to replace the definition of \( f_n' \) by setting 
\begin{equation*}
f_n'(x) = \begin{cases} 1 &\textnormal{if } x(1) = \ldots = x({k(n)}) = 1 \cr f_n(x) &\textnormal{else}\end{cases}
\end{equation*}
instead.
\end{remark}

The rest of this paper will be structured as follows. In Section~2 we give a proof of Theorem~\ref{theorem: striped example} and  in Section~3 we give a proof of Theorem~\ref{theorem: noise stable and volatile}.

\section{Proof of Theorem~\ref{theorem: striped example}}

In this section we will give  a proof of Theorem~\ref{theorem: striped example}, that is we will  show that if
\[
\lim_{n \to \infty} np_n(1-p_n) = \infty,
\]
then any sequence of Boolean functions \( f_n \colon \{ -1,1 \}^n \to \{ -1,1 \} \) is \( o(1) \)-close to a  sequence of Boolean functions which is volatile with respect to \( (\pi_{p_n}) \).

The main idea in the proof of this result is to, for each \( n \), construct a function \( g_n \) by changing the value of \( f_n \) on certain level sets. By level sets, we mean sets of the form \( \{ x \in \{ -1,1 \}^n \colon \| x \| = \ell \} \) for some  value of \( \ell \), where \( \| x \| \coloneqq \sum_i \frac{x(i) + 1}{2} \). If these level sets are sparse enough, then \( (f_n) \) and \( (g_n) \) will be \( o(1) \)-close, and if the level sets are at the same time close enough, then the Markov chain \( (X_t^{(n)}) \)  will almost surely hit two adjacent such level sets very quickly, and hence be volatile. If we set the value of \( g_n  \) to \( 1 \) on every second such level set and to \( -1 \) on the other level sets, then this would guarantee that \( (g_n) \) is volatile.

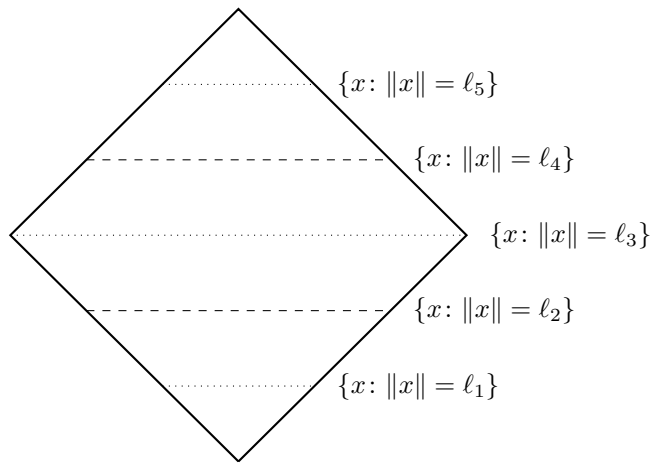
\begin{figure}[htp]
\centering
\begin{tikzpicture}
\draw[thick] (0,0) -- (3,3) -- (0,6) -- (-3,3) -- (0,0);
\draw[dotted] (-1,1) -- (1,1) node[right, xshift=0.5em] {\( \{ x \colon \| x \| = \ell_1 \} \)};
\draw[dashed] (-2,2) -- (2,2) node[right, xshift=0.5em] {\( \{ x \colon \| x \| = \ell_2 \} \)};
\draw[dotted] (-3,3) -- (3,3) node[right, xshift=0.5em] {\( \{ x \colon \| x \| = \ell_3 \} \)};
\draw[dashed] (-2,4) -- (2,4) node[right, xshift=0.5em] {\( \{ x \colon \| x \| = \ell_4 \} \)};
\draw[dotted] (-1,5) -- (1,5) node[right, xshift=0.5em] {\( \{ x \colon \| x \| = \ell_5 \} \)};
\end{tikzpicture}
\caption{The figure above depicts the Hamming cube \( \{ -1,1 \}^n \) with some level sets. The dashed lines represents level sets where the function \( g_n \), as defined in~\eqref{eq: def of g_n}, is equal to \( 1 \), and the dotted lines represents level sets where the value of \( g_n \) is equal to \( -1 \). }
\end{figure}

In the proof of this theorem, as well as in the proof of Theorem~\ref{theorem: noise stable and volatile}, we will use well known results on the distribution of Ornstein-Uhlenbeck processes. An Ornstein-Uhlenbeck process \((Z_t) \) with infinitesimal mean \( \mu(z) = -z\) and infinitesimal variance \( \sigma^2(z) = 1\) is defined as the solution to the stochastic differential equation 
\[
dZ_t = -Z_t \, dt + \sqrt{2} dW_t,
\]
where \( W_t \) is a Wiener process.
This stochastic process is useful to us at it arises as the limit of \( \| X_t^{(n)} \| \) as \( n \) tends to infinity, after a suitable normalisation given in the first part of the lemma below. The results in this lemma are well know, and can be found e.g. in~\cite{kt1981}, pp.170--173.

\begin{lemma}\label{lemma: OU process}
Suppose that \( \lim_{n \to \infty} np_n(1-p_n) = 0 \) and let \( (Z_t) \) be an Ornstein-Uhlenbeck process with infinitesimal mean \( \mu(z) = -z \) and infinitesimal variance \( \sigma^2(z) = 1\). Then
\begin{equation*}
\left( \frac{\| X_t^{(n)} \| - np_n}{\sqrt{2np_n(1-p_n)}}  \right)_{ t \geq 0 }
\end{equation*}
converges to \( (Z_t)_{t \geq 0 } \) in distribution. Moreover, for such an Ornstein-Uhlenbeck process \( (Z_t) \), we have that
\begin{equation*}\label{eq: OU mean}
\E \left[ Z(t+s) \mid Z(t) = z \right] = ze^{-s},\quad s,t \geq 0
\end{equation*}
and
\begin{equation*}\label{eq: OU variance}
\Var [Z(t+s) \mid Z(t) = z]= \frac{1-e^{-2s}}{2},\quad s,t \geq 0.
\end{equation*}
\end{lemma}

\begin{proof}[Proof of Theorem~\ref{theorem: striped example}]

 For \( n \in \mathbb{N}\), define  \( (\alpha_n) \)  by 
\[
\alpha_n \coloneqq   \left( 2np_n(1-p_n) \right)^{1/4}
\]
and note that by assumption, 
\begin{equation}\label{eq: alpha convergence}
 \lim_{n \to \infty} \alpha_n = \infty .
\end{equation}
Now note that
\begin{align*}
&\left| \| X_0^{(n)} \| -\| X_t^{(n)} \| \right| > 2\alpha_n 
\\&\qquad \Leftrightarrow 
\left| \frac{\| X_0^{(n)} \|-np_n}{\sqrt{2np_n(1-p_n)}} -\frac{\| X_t^{(n)} \|-np_n}{\sqrt{2np_n(1-p_n)}} \right| > \frac{2\alpha_n}{\sqrt{2np_n(1-p_n)}} = 2\alpha^{-1}.
\end{align*}
Combining this with~\eqref{eq: alpha convergence} and Lemma~\ref{lemma: OU process}, we obtain that  for any fixed \( t > 0 \) we have that
\[
\lim_{n \to \infty} P\left[   \left| \| X_0^{(n)} \| -\| X_t^{(n)} \| \right| > 2\alpha_n  \right]= 1.
\]
This implies in particular that for all \( \varepsilon > 0 \),
\begin{equation}\label{eq: distance thing}
\lim_{n \to \infty} P\left[\sup_{t \in (0,\varepsilon)} \left| \| X_0^{(n)} \| -\| X_t^{(n)} \| \right| > 2\alpha_n  \right]= 1.
\end{equation}

Now given a sequence \( f_n \colon \{ -1,1 \}^n \), for each \( n \) let  \( a_n < \alpha_n \) be  a non-negative integer and set 
\[
 L_n \coloneqq \{ a_n + i \cdot \alpha_n  ,\, i=0,1,2, \ldots \} \cap \{ 0,1, \ldots, n \} .
\]
For notational simplicity we assume that \(\alpha_n   \) is an integer.
Now note that for each \( n \), there are exactly \( \alpha_n  \) possible ways to choose \( a_n \), and the sets \( L_n \) we get for different \( a_n  \) form a  partition of \( [n] \). Consequently,  there will be at  least one way to choose  \( (a_n) \) such that
\begin{equation}\label{eq: example equation}
P \left[   \| X_0^{(n)} \| \in L_n   \right] \leq \alpha_n^{-1}
\end{equation}
for all \(n \).
Fix such a sequence \( (a_n) \) and for \( x \in \{ -1,1 \}^n \) define
\begin{equation}\label{eq: def of g_n}
g_n(x) \coloneqq \begin{cases} f_n(x) & \text{if } \| x \| \not \in L_n \cr  1 & \text{if } \| x \|  = a_n + i \cdot \alpha_n \text{ for }  i \textnormal{ odd}    \cr  -1 & \text{if } \| x \| = a_n + i \cdot \alpha_n \text{ for }  i \textnormal{ even}. \end{cases}
\end{equation}
Then
\[
 P\left[f_n(X_0) \not = g_n(X_0)\right] \leq \alpha_n^{-1}
 \]
 and hence  \( (f_n) \) and \( (g_n) \) are \( o(1) \)-close by~\eqref{eq: alpha convergence}~and~\eqref{eq: example equation}. Finally, it follows from~\eqref{eq: distance thing}~and~\eqref{eq: def of g_n} that \( (g_n) \) is volatile.

\end{proof}

\section{Proof of Theorem~\ref{theorem: noise stable and volatile}}

The proof of Theorem~\ref{theorem: noise stable and volatile} will be divided over the following two subsections. In the first subsection, we will prove that there is a symmetric, monotone, nondegenerate and  noise stable sequence of Boolean functions that is not \( o(1)\)-close to any non-volatile sequence of Boolean functions. In the second subsection we will prove that there is a nondegenerate and noise stable sequence of Boolean functions that are invariant under permutations of the coordinates, that is not \( o(1)\)-close to any non-volatile sequence of Boolean functions.

\subsection{Proof of the first part of Theorem~\ref{theorem: noise stable and volatile} (monotone and symmetric)}

Let \((\ell_i)_{i \geq 1}\) be a strictly increasing sequence of positive integers with \(   \ell_1 = 1 \), and for \( i = 1,2, 3, \ldots \) define sets \( S_i \subseteq \mathbb{N} \) by    
\[
 S_i =   \{ 1 + \sum_{j = 1}^{i-1} \ell_j ,2 + \sum_{j = 1}^{i-1} \ell_j , \ldots, \ell_i + \sum_{j = 1}^{i-1} \ell_j  \} .
\]
For \( x \in \{ -1,1 \}^n \), let \( i(x) \) be the largest index \( i \) such that
\( x(S_j) = x(S_k) \) for all \( j,k \in S_i \) and
 \( \sum_{j = 1}^i \ell_j \leq n \). Define \( f_n(x) \) by
\[
f_n(x) =
\sgn(x( S_{i(x)})).
\]
In other words, let \( f_n(x) \) be the sign of \( x \) in the last constant block \( S_{i(x)} \).

We will   prove the following proposition, from which the first part of Theorem~\ref{theorem: noise stable and volatile} follows.

\begin{proposition}\label{proposition: block fcn result}
Let \( (\ell_i)\) and \( (f_n)\) be as above. Then \( (f_n) \) is symmetric, nondegenerate and  monotone. Further, if
\begin{equation}\label{eq: bounded sum}
\sum 2^{-\ell_i} \ll \infty
\end{equation}
and 
\begin{equation}\label{eq: infinite sum}
\sum \ell_i 2^{-\ell_i} = \infty.
\end{equation}
then \( (f_n) \) is noise stable, volatile and not \( o(1) \)-close to any non-volatile sequence of functions.
\end{proposition}

In the proof of Proposition~\ref{proposition: block fcn result} we will use the following lemma, whose proof is essentially  due to Anders Martinsson.
\begin{lemma}\label{lemma: Anders lemma}
Let \( \tau \) be the hitting time of the state \( x = (1,1,\ldots,1)\in \{ -1,1\}^n \) and let \( \delta > 0 \). Then
\begin{equation*}
P[\tau \leq \delta] \geq  \frac{1 - e^{-\delta}}{2} \cdot  n 2^{-n}.
\end{equation*}
\end{lemma}

\begin{proof}

Let \( 1^n \) denote the element \( (1,1,\ldots,1) \in \{ -1,1 \}^n\). Then
\begin{equation*}
\begin{split}
\delta 2^{-n} 
&= \E\left[  \int_0^\delta \mathbf{1}_{X_t =1^n} \, dt \right]
= P[\tau \leq \delta] \cdot \E\left[  \int_0^\delta \mathbf{1}_{X_t = 1^n} \, dt \mid \tau \leq \delta \right]
\\
&= 
P[\tau \leq \delta] \cdot \E\left[  \int_\tau^\delta \mathbf{1}_{X_t = 1^n} \, dt \mid \tau \leq \delta \right]
\leq
P[\tau \leq \delta] \cdot \E\left[  \int_0^\delta \mathbf{1}_{X_t = X_0} \, dt  \right]
\\&= P[\tau \leq \delta] \cdot \int_0^\delta P[X_t =X_0]\, {d}t
= P(\tau \leq \delta) \cdot \int_0^\delta \left(\frac{1 + e^{-t}}{2} \right)^n  \, {d}t
\\&= P[\tau \leq \delta] \cdot \frac{1}{n} \int_0^{\delta n} \left(\frac{1 + e^{-s/n}}{2} \right)^n  \, {d}s.
\end{split}
\end{equation*}
For \( s \leq \delta n \) we have that
\begin{equation*}
\begin{split}
&\left(\frac{1 + e^{-s/n}}{2} \right)^n 
=
\left( 1 - \frac{1 - e^{-s/n}}{2} \right)^n 
=
\left( 1 - \frac{s}{2n} \cdot \frac{1 - e^{-s/n}}{s/n} \right)^n 
\\&\qquad\leq
\left( 1 - \frac{s}{2n} \cdot \frac{1 - e^{-\delta}}{\delta} \right)^n 
\leq
\exp\left(  - \frac{s}{2} \cdot \frac{1 - e^{-\delta}}{\delta} \right)
\end{split}
\end{equation*}
and hence 
\begin{equation*}
\begin{split}
& \int_0^{\delta n} \left(\frac{1 + e^{-s/n}}{2} \right)^n  \, {d}s
\leq
 \int_0^\infty\exp\left(  - \frac{s}{2} \cdot \frac{1 - e^{-\delta}}{\delta} \right)  \, {d}s
\\
&\qquad=
 \left[ -\frac{2\delta} {1 - e^{-\delta}}\exp\left(  - \frac{s}{2} \cdot \frac{1 - e^{-\delta}}{\delta} \right)   \right]_0^\infty 
=
\frac{2\delta} {1 - e^{-\delta}}.
\end{split}
\end{equation*}
Rearranging, we  obtain
\[
P\left[ \tau \leq \delta \right] \geq   \frac{1 - e^{-\delta}}{2} \cdot  n 2^{-n}
\]
which is the desired conclusion.

\end{proof}

We now give a proof of Proposition~\ref{proposition: block fcn result}.

\begin{proof}[Proof of Proposition~\ref{proposition: block fcn result}]
Note first that symmetry and monotonicity is immidiate from the definiton, and that any symmetric function is nondegenerate with respect to \( p_n = 1/2\).

To see that \( (f_n) \) is volatile,    note  first that if we define
\begin{equation*}
\mu_n(i) \coloneqq P\left[S_i \textnormal{ is the last constant block  in } X_0\right] 
\end{equation*}
then
\begin{equation*}
\mu_n(i)  = 2 \cdot 2^{-\ell_i}  \prod_{j = i+1} (1 - 2 \cdot  2^{-\ell_i}) \leq 2 \cdot 2^{-\ell_i}
\end{equation*}
and hence it follows from~\eqref{eq: bounded sum} that \( ( \mu_n ) \) is tight. 

Now for any \( m \in \mathbb{N}\), let \( 1^m \) denote the element \( (1,1,\ldots,1) \in \{-1,1\}^m\). Then for any fixed \( i \in \mathbb{N} \) we have that
\begin{equation*}
\begin{split}
&\lim_{n \to \infty} P \left[ \textnormal{there is no } j>i \textnormal{ and } t<\delta \textnormal{ such that } X_t(S_j) =1^{|S_j|}  \right]
\\&\qquad\leq \lim_{n \to \infty}  \prod_{j > i} \left(1 - \frac{1-e^{-\delta}}{2} \cdot  \ell_j 2^{-(\ell_j+1)} \right)
\\& \qquad \leq \lim_{n \to \infty}  \exp \left(  -  \frac{1-e^{-\delta}}{2}\cdot   \sum_{j > i}  \ell_j 2^{-(\ell_j+1)} \right)
= 0
\end{split}
\end{equation*}
where the first inequality is obtained by  using Lemma~\ref{lemma: Anders lemma} and the limit  follows from~\eqref{eq: infinite sum}. Together with the fact that \( ( \mu_n ) \) is tight, this implies that \( (f_n) \) is volatile.

To see that \( (f_n) \) is noise stable, note first that
\begin{equation*}
\begin{split}
& P\left[f_n(X_\varepsilon) \not = f_n (X_0) \mid i(X_0)=i \right] 
\\[1ex]&\qquad  \leq
P\left[X_\varepsilon (S_i) \not \equiv 1 \mid X_0(S_i) \equiv 1 \right] 
\\&\qquad\qquad + \sum_{j > i} P\left[X_\varepsilon(S_j)) \equiv -1 \mid X_0(S_j) \not \in \{ (-1)^n, 1^n \}\right]
\\[-0.5ex]&\qquad  = 
 \left( 1  - \left( \textstyle\frac{1 + e^{-\varepsilon}}{2} \right)^{\ell_i} \right) + \sum_{j > i} \textstyle\frac{2 \cdot 2^{-\ell_j}}{1 - 2\cdot 2^{-\ell_j}} \left( 1 - \left( \textstyle\frac{1 - e^{-\varepsilon}}{2} \right)^{\ell_j} - \left( \textstyle\frac{1+ e^{-\varepsilon}}{2} \right)^{\ell_j} \right)
\\&\qquad  \leq 
 \left( 1  - \left( \textstyle\frac{1 + e^{-\varepsilon}}{2} \right)^{\ell_i} \right) + \sum_{j > i} \textstyle\frac{2 \cdot 2^{-\ell_j}}{1 - 2\cdot 2^{-\ell_j}} \left( 1 - \left( \textstyle\frac{1+ e^{-\varepsilon}}{2} \right)^{\ell_j} \right) 
\\&\qquad  \leq 
 \left( 1  - \left( \textstyle\frac{1 + e^{-\varepsilon}}{2} \right)^{\ell_i} \right) + 4\sum_{j \geq 2}  2^{-\ell_j} \left( 1 - \left( \textstyle\frac{1+ e^{-\varepsilon}}{2} \right)^{\ell_j} \right) .
\end{split}
\end{equation*}
Using~\eqref{eq: bounded sum} and the fact that \( (\mu_n) \)  is tight, it now follows directly that \( (f_n) \) is noise stable.

It now remains only to argue that if  \( g_n \colon \{ -1,1 \}^n \to \{ -1,1 \} \) and \( (f_n) \) and \( (g_n) \) are \( o(1)\)-close, then \( (g_n)\) must be volatile.
To this end, fix  \( \delta > 0\) and let \( \varepsilon > 0\) be chosen arbitrarily. 
We can then  pick \( N = N(\varepsilon,\delta) \)  such that
\begin{enumerate}
	\item[(i)] \( P\left[ \tau_{\partial f_n}>\delta \right] < \varepsilon \) for all \( n > N\) (this follows from the volatility of \( (f_n) \)).
	\item[(ii)] \(P\left[ \max S_{i(X_0)} > m \right] < \varepsilon \) for all \( m > M \) (this follows from the tightness of \( (\mu_n)\))
	\item[(iii)]  For all \( m > N\) and \( n \geq m \), \( P \left[ f_m(X_0([m])) \not =  f_n(X_0) \mid X_0([m]) = x([m])\right] < \varepsilon  \)  
\end{enumerate}
Now fix \( m > N\). Then
\begin{align*}
&\liminf_{n \to \infty} P \left[ 
\tau_{\partial g_n} < \delta \right] 
\\&\qquad\geq 
\liminf_{n \to \infty}P \left[ 
\begin{matrix} \tau_{\partial f_m} < \delta,\, 
\max S_{i(X_0)}<m,\,
f_n(X_0)=g_n(X_0) \\[1ex]
\textnormal{and }
f_m(X_{\tau_{\partial f_m}}([m]))=f_n(X_{\tau_{\partial f_m}})=g_n(X_{\tau_{\partial f_m}})
\end{matrix}
\right] 
 \\&\qquad=
\liminf_{n \to \infty}P \left[ 
\tau_{\partial f_m} < \delta,\, 
\max S_{i(X_0)}<m \textnormal{ and }
f_m(X_{\tau_{\partial f_m}}([m]))=f_n(X_{\tau_{\partial f_m}})=g_n(X_{\tau_{\partial f_m}})
\right] 
 \\&\qquad\overset{(i)}{\geq}
 \liminf_{n \to \infty}P \left[ 
 \max S_{i(X_0)}<m \textnormal{ and }
 f_m(X_{\tau_{\partial f_m}}([m]))=f_n(X_{\tau_{\partial f_m}})=g_n(X_{\tau_{\partial f_m}})
 \right] 
 -\varepsilon 
\\&\qquad\overset{(ii)}{\geq}
\liminf_{n \to \infty}P \left[ 
f_m(X_{\tau_{\partial f_m}}([m]))=f_n(X_{\tau_{\partial f_m}})=g_n(X_{\tau_{\partial f_m}})
\right] 
-2\varepsilon 
\\&\qquad\overset{(iii)}{\geq}
\liminf_{n \to \infty}P \left[ 
f_n(X_{\tau_{\partial f_m}})=g_n(X_{\tau_{\partial f_m}})
\right] 
-3\varepsilon
\end{align*}
As \( \tau_{\partial f_m}\)  only depends on the first \( m \) coordinates of \( X_t \), and \( (f_n) \) and \( (g_n) \) are \( o(1)\)-close, we obtain
\[
\liminf_{n \to \infty} P \left[\tau_{\partial g_n} < \delta \right] \geq 1 - 3 \varepsilon.
\]
As \( \varepsilon > 0 \) was arbitrary, it follows that \( (g_n)\) is volatile.

\end{proof}

\clearpage

\subsection{Proof of the second part of Theorem~\ref{theorem: noise stable and volatile} (permutation invariance)}

The goal of this subsection is to give an example of a nondegenerate, noise stable and volatile sequence of Boolean functions \( f_n \colon \{ -1,1 \}^n \to \{ -1,1 \} \) that is not \( o(1) \)-close to any non-volatile sequence of Boolean functions. We will do this by  first constructing a sequence of functions \( f_n \colon \mathbb{Z}_n \to \{ -1,1 \}\), where \( \mathbb{Z}_n = \mathbb{Z}/n\mathbb{Z} \), with the analogous properties for \( \mathbb{Z}_n \) and then translating this construction to the hypercube.
To be able to do this we must first extend the definitions of noise stability and volatility to the setting of continuous time simple random walks  on \( \mathbb{Z}_n \). To this end, define \( Y_t^{(n)} \) to be  continuous time simple random walk on \( \mathbb{Z}_n \) which takes steps with rate \( n^2 \). The rate \( n^2 \) is chosen so that the relaxation time for this Markov chain is the same as the relaxation time for the random walk \( (X_t^{(n)}) \) on the hypercube \( \{ -1,1\}^n \). The stationary distribution for this Markov chain is the uniform distribution on \( \mathbb{Z}_n \) and we will always choose \( Y_0^{(n)} \) according to this distribution. We now define what we mean by being noise stable and volatile in this setting.

\begin{definition}
A sequence of functions \( f_n \colon \mathbb{Z}_n \to \{ -1,1 \} \) is said to be \emph{noise stable} if
\[
\lim_{\varepsilon \to 0 } \sup_n P \left[ f_n(Y_0^{(n)}) \not = f_n(Y_{\varepsilon}^{(n)}) \right] = 0
\]
\end{definition}

\begin{definition}
A sequence of functions \( f_n \colon \mathbb{Z}_n \to \{ -1,1 \} \) is said to be \emph{volatile} if for all \( \delta > 0 \),
\[
\lim_{n \to \infty} P\left[\tau_{\partial f_n} > \delta \right]= 0
\]
where \( \tau_{\partial f_n} \) is the hitting time of the set \( \{ x \in \mathbb{Z}_n \colon f_n(x) \not = f_n(Y_0) \} \).
\end{definition}

Using these two definitions, we can formulate the following proposition.

\begin{proposition}\label{proposition: analogue on circle}
There is a nondegenerate, noise stable and volatile sequence of functions \( f_n \colon \mathbb{Z}_n \to \{ -1,1 \} \) that is not  \( o(1) \)-close to any non-volatile sequence of functions  \( f_n \colon \mathbb{Z}_n \to \{ -1,1 \} \).
\end{proposition}

\begin{proof}[Proof of Proposition~\ref{proposition: analogue on circle}]

We first construct sets of functions \( \{ \mathcal{F}_i \}_{i \in \mathbb{N}} \) as follows. 
\begin{itemize}
\item Let \( I_0^1 = C^1 \). We say that  \( f \colon C^1 \to \{ -1,1 \} \) is in \( \mathcal{F}_1 \) if 
\[
  P\left[f(Y) = 1  \right]  = P\left[f(Y) = 1 \mid Y \in I_0^1 \right] = 2^{-1},
\] where \( Y \sim \unif(C^1) \).  We say that \( \mathcal{F}_1 \) is defined by the interval \( I_0^1 \) with associated density \( 2^{-1} \).

\item Let \( I_0^2 = [0,0.5) \) be the upper half of the unit circle and let \( I_1^2 = [0.5,1) \) be the lower half. We say that \( f \colon C^1 \to \{ -1,1 \} \) is in \( \mathcal{F}_2 \) if
\[
 P\left[f(Y) = 1 \mid Y \in I_0^2 \right] = 2^{-2}
\]
and 
\[
 P\left[f(Y) = 1 \mid Y \in I_1^2 \right] = 1-2^{-2}.
\]
\( \mathcal{F}_2 \) is thus defined by the intervals \( I_0^2 \) and \( I_1^2 \) with associated densities \( 2^{-2} \) and \( 1-2^{-2} \). 
Note that \( I_0^2 \) and \( I_1^2 \) form a partition of \( C^1 \) and that \( \mathcal{F}_2 \subseteq \mathcal{F}_1 \).

\item For \( k \geq 3 \), given \( \mathcal{F}_{k-1} \) we  construct \( \mathcal{F}_{k} \) as follows. For each interval \( I_i^{k-1} \) with associated density \( 2^{-(k-1)} \), pick a centered subinterval \( I_{3i+1}^{k} \subseteq I_i^{k-1}\) with length \(  \frac{2^{-{k}}}{1 - 2^{-(k-1)}} \cdot \left| I_i^{k-1} \right|\) and associated density \( 1 - 2^{-k} \). Let the two remaining intervals be called \( I_{3i}^{k}\) and \( I_{3i+2}^{k}\) and have associated densities  \( 2^{-k} \). Similarly, for each interval \( I_i^{k-1} \) with associated density \( 1-2^{-(k-1)} \), pick a centered interval \( I_{3i+1}^{k}\) with length  \(  \frac{2^{-k}}{1 - 2^{-(k-1)}} \cdot \left| I_i^{k-1} \right|\) and associated density \( 2^{-{k}} \) and let the remaining intervals be called \( I_{3i}^{k}\) and \( I_{3i+2}^{k}\) and have associated densities  \( 1-2^{-k} \). Let \( f \in  \mathcal{F}_{k} \) if for each such interval \(I^{k}_i \) with associated density \( p^{k}_i \) we have that  
\begin{equation}\label{eq: general prob for intervals}
 P\left[f(Y) = 1 \mid X \in I^{k}_i \right] = p^{k}_i.
\end{equation}
Note that with this choice of interval lengths, we have \( \mathcal{F}_{k+1} \subseteq \mathcal{F}_k \).
\end{itemize}

\begin{figure}[htp]
\begin{minipage}[t]{0.42\linewidth}
\centering
        \begin{tikzpicture}
		\draw [black!32,line width=1.5mm,domain=0:360, samples=101] plot ({1.5*cos(\x)}, {1.5*sin(\x)});
	\end{tikzpicture}
\caption*{(a) All functions   \( f \in  \mathcal{F}_1 \) has \( P[f(X_0) = 1 ] = 2^{-1}\). }
\end{minipage}
\hspace{0.5cm}
\begin{minipage}[t]{0.42\linewidth}
\centering
        \begin{tikzpicture}
		\draw [black!16,line width=1.5mm,domain=0:180] plot ({4+1.5*cos(\x)}, {1.5*sin(\x)});
		\draw [black!48,line width=1.5mm,domain=180:360] plot ({4+1.5*cos(\x)}, {1.5*sin(\x)});
	\end{tikzpicture}
\caption*{(b) All functions \(f \in \mathcal{F}_2 \) has density \(  2^{-2} \) on the upper half of the circle, and density \(  1 -2^{-2} \) on the lower half.}
\end{minipage}
\vspace{6ex}

\begin{minipage}[t]{0.42\linewidth}
\centering
        \begin{tikzpicture}
		\draw [black!8,line width=1.5mm,domain=0:75] plot ({1.5*cos(\x)}, {-4+1.5*sin(\x)});
		\draw [black!56,line width=1.5mm,domain=75:105] plot ({1.5*cos(\x)}, {-4+1.5*sin(\x)});
		\draw [black!8,line width=1.5mm,domain=105:180] plot ({1.5*cos(\x)}, {-4+1.5*sin(\x)});
		\draw [black!56,line width=1.5mm,domain=180:255] plot ({1.5*cos(\x)}, {-4+1.5*sin(\x)});
		\draw [black!8,line width=1.5mm,domain=255:285] plot ({1.5*cos(\x)}, {-4+1.5*sin(\x)});
		\draw [black!56,line width=1.5mm,domain=285:360] plot ({1.5*cos(\x)}, {-4+1.5*sin(\x)});
	\end{tikzpicture}
\caption*{(c) All functions \(f \in \mathcal{F}_3 \) has density \( 1 - 2^{-3}  \) in the darker segments of the circle above, and density \( 2^{-3} \) in the lighter segments.}
\end{minipage}
\hspace{0.5cm}
\begin{minipage}[t]{0.42\linewidth}
\centering
        \begin{tikzpicture}
		\draw [black!4,line width=1.5mm,domain=0:35] plot ({4+1.5*cos(\x)}, {-4+1.5*sin(\x)});
		\draw [black!60,line width=1.5mm,domain=35:40] plot ({4+1.5*cos(\x)}, {-4+1.5*sin(\x)});
		\draw [black!4,line width=1.5mm,domain=40:75] plot ({4+1.5*cos(\x)}, {-4+1.5*sin(\x)});
		\draw [black!60,line width=1.5mm,domain=75:89] plot ({4+1.5*cos(\x)}, {-4+1.5*sin(\x)});
		\draw [black!4,line width=1.5mm,domain=89:91] plot ({4+1.5*cos(\x)}, {-4+1.5*sin(\x)});
		\draw [black!60,line width=1.5mm,domain=91:105] plot ({4+1.5*cos(\x)}, {-4+1.5*sin(\x)});
		\draw [black!4,line width=1.5mm,domain=105:140] plot ({4+1.5*cos(\x)}, {-4+1.5*sin(\x)});
		\draw [black!60,line width=1.5mm,domain=140:145] plot ({4+1.5*cos(\x)}, {-4+1.5*sin(\x)});
		\draw [black!4,line width=1.5mm,domain=145:180] plot ({4+1.5*cos(\x)}, {-4+1.5*sin(\x)});

		\draw [black!60,line width=1.5mm,domain=0:-35] plot ({4+1.5*cos(\x)}, {-4+1.5*sin(\x)});
		\draw [black!4,line width=1.5mm,domain=-35:-40] plot ({4+1.5*cos(\x)}, {-4+1.5*sin(\x)});
		\draw [black!60,line width=1.5mm,domain=-40:-75] plot ({4+1.5*cos(\x)}, {-4+1.5*sin(\x)});
		\draw [black!4,line width=1.5mm,domain=-75:-89] plot ({4+1.5*cos(\x)}, {-4+1.5*sin(\x)});
		\draw [black!60,line width=1.5mm,domain=-89:-91] plot ({4+1.5*cos(\x)}, {-4+1.5*sin(\x)});
		\draw [black!4,line width=1.5mm,domain=-91:-105] plot ({4+1.5*cos(\x)}, {-4+1.5*sin(\x)});
		\draw [black!60,line width=1.5mm,domain=-105:-140] plot ({4+1.5*cos(\x)}, {-4+1.5*sin(\x)});
		\draw [black!4,line width=1.5mm,domain=-140:-145] plot ({4+1.5*cos(\x)}, {-4+1.5*sin(\x)});
		\draw [black!60,line width=1.5mm,domain=-145:-180] plot ({4+1.5*cos(\x)}, {-4+1.5*sin(\x)});
	\end{tikzpicture}
\caption*{(d) All functions \(f \in \mathcal{F}_4 \) has has density \( 1 - 2^{-4}  \) in the darker segments of the circle above, and density \( 2^{-4} \) in the lighter segments.}
\end{minipage}
\caption{The four figures above shows the first four steps in the construction of \( \{ \mathcal{F}_i \}_{i \in \mathbb{N}} \).}
\end{figure}

For \( k \in \mathbb{N} \), define \( \mathcal{F}_k' \) by saying  that a function \( f \colon \mathbb{Z}_n \to \{-1,1 \}  \in \mathcal{F}_k' \) if for each interval \( I_i^k \) with associated density \( p_i^k \) we have that 
\begin{equation}
P \left[f(Y_0) = 1 \mid {\textstyle \frac{1}{n}}    Y_0 \in I_i^k \right] \in \left( p_i^k - 2^{-(k+1)}, p_i^k + 2^{-(k+1)} \right).\label{eq: mean on interval}
\end{equation}
Recall that \( p_i^k \in \{ 2^{-k}, 1-2^{-k} \} \), hence the interval in the right hand side of the previous equation is either \( (2^{-(k+1)}, 3\cdot 2^{-(k+1)}) \) or \( (1 - 3 \cdot 2^{-(k+1)}, 1 - 2^{-(k+1)} )\). One can verify that there is at least one \( f \colon \mathbb{Z}_n \to \{ -1,1 \} \in \mathcal{F}_k' \) for \( n \geq 2^{k^2+k}  \).

Now let \((f_n) \), \( f_n \colon \mathbb{Z}_n \to \{ -1,1 \} \),  be a sequence with \( f_n \in \mathcal{F}_{a_n}' \) for some sequence \( a_n \to \infty \). We will show that \( (f_n) \) is nondegenerate w.r.t. \(p_n = 1/2\), noise stable and volatile, but not \( o(1)\)-close to any non-volatile function.

\begin{claim}\label{claim: nondegeneracy} \( (f_n) \) is nondegenerate. \end{claim}

\begin{proof}
To see that \( (f_n) \) is nondegenerate with respect to \( p_n = 0.5 \), simply note that
\begin{equation*}\label{eq: nondeg 1}
\begin{split}
P\left[f_n(Y_0) = 1\right] &= \sum_{i}P\left[{\textstyle \frac{1}{n}}  Y_0  \in I_i^{a_n}\right] \cdot P\left[f_n(Y_0) = 1 \mid \textstyle \frac{1}{n}  Y_0 \in I_i^{a_n}\right]
\\&\geq \sum_{i}P\left[{\textstyle \frac{1}{n} }  Y_0 \in I_i^{a_n}\right] \cdot (p_i^{a_n} - 2^{-(a_n+1)}) 
\\&= \sum_{i}P({\textstyle \frac{1}{n}}  Y_0 \in I_i^{a_n}) \cdot p_i^{a_n} - \sum_{i}P\left[{\textstyle \frac{1}{n} }  Y_0 \in I_i^{a_n}\right] \cdot  2^{-(a_n+1)}
\\&= 0.5 -  2^{-(a_n+1)}
\end{split}
\end{equation*}
and that analogously,
\begin{align*}
P\left[f_n(Y_0) = 1\right] &= \sum_{i}P\left[{\textstyle \frac{1}{n} }  Y_0 \in I_i^{a_n}\right] \cdot P\left[f_n(Y_0) = 1 \mid {\textstyle \frac{1}{n} }  Y_0 \in I_i^{a_n}\right]
\\&\leq \sum_{i}P\left[{\textstyle \frac{1}{n} }  Y_0\in I_i^{a_n}\right] \cdot (p_i^{a_n} + 2^{-(a_n+1)}) 
\\&= \sum_{i}P\left[{\textstyle \frac{1}{n} }  Y_0 \in I_i^{a_n}\right] \cdot p_i^{a_n} + \sum_{i}P\left[{\textstyle \frac{1}{n} }  Y_0 \in I_i^{a_n}\right] \cdot  2^{-(a_n+1)}
\\&= 0.5 +  2^{-(a_n+1)}.
\end{align*}
Consequently, \( (f_n) \) is nondegenerate.
\end{proof}

\begin{claim}\label{claim: noise stability}
\( (f_n) \) is noise stable.
\end{claim}

\begin{proof}
Let \( k \in \mathbb{N} \) be arbitrary and suppose that \( n \) is large enough to have \( a_n > k \). Let \( \ell = \min_i |I_i^k| \). By the Central limit theorem, for all \( \delta > 0 \) we have that
\[
\lim_{\varepsilon \to 0} \limsup_{n \to \infty} P\left[\left| {\textstyle \frac{1}{n}}  Y_0 - {\textstyle \frac{1}{n}}  Y_{\varepsilon} \right| > \delta \ell \right] = 0.
\]
From this it follows that
\[
\lim_{\varepsilon \to 0} \liminf_{n \to \infty}  P \left[\exists i \colon \left\{ {\textstyle \frac{1}{n}}  Y_0, {\textstyle \frac{1}{n}}  Y_{\varepsilon } \right\} \in I_i^k \right] = 1
\]
and consequently, using~\eqref{eq: mean on interval}, we obtain
\[
\lim_{\varepsilon \to 0} \limsup_{n \to \infty}  P \left[ f_n(Y_0) \not = f_n(Y_{\varepsilon}) \right] \leq 3 \cdot 2^{-k}.
\]
 As \( k \) was arbitrary, we can conclude that \( (f_n) \) is noise stable.
\end{proof}

\begin{claim}\label{claim: volatility}
\( (f_n) \) is volatile.
\end{claim}

\begin{proof}

Note first that \( \left| I_i^k \right| \leq 2^{-(k-1)} \) for all \( i \) and \( k \). Also, if \( a_n > k \) we have that \( \mathcal{F'}_{a_n} \subseteq \mathcal{F}'_{k} \). Moreover, as \( p_i^k \pm 2^{-(k+1)} \in (0,1) \) for all \( i \) and \( k \), if all states \( x \in \mathbb{Z} \) such that \( {\textstyle \frac{1}{n}} \cdot x \in I_i^k \) has been visited for some \( i \) and \( k \), then at least one state with \( f_n(x) = 1 \) and at least one state with \( f_n(x) = -1 \) must have been visited.
Consequently, if \( a_n > k \) we have that
\[
	P\left[\tau_{\partial f_n} > \delta \right] 
	\leq  P \left[ \left| {\textstyle \frac{1}{n}}  Y_0 - {\textstyle \frac{1}{n}}  Y_{\delta } \right| < 2 \cdot 2^{-(k-1)} \right].
\]
and as \( a_n \to \infty \) by assumption, it follows that
\[
	\limsup_{n \to \infty} P\left[\tau_{\partial f_n} > \delta \right] 
	\leq  \limsup_{n \to \infty} P\left[ \left| {\textstyle \frac{1}{n}}  Y_0 - {\textstyle \frac{1}{n}}  Y_{\delta } \right|< 2\cdot 2^{-(k-1)} \right]
\]
for all \( k \in \mathbb{N} \). By the Central limit theorem, the right hand side of the last equation can be made arbitrarily small by choosing \( k \) to be large, and hence \( (f_n) \) must be volatile.
\end{proof}

\begin{claim}
\( (f_n) \) is not \( o(1) \)-close to any nonvolatile sequence \( f_n' \colon \mathbb{Z}_n \to \{ -1,1 \} \).
\end{claim}

\begin{proof}
To that the claim is true, let \( k \) be arbitrary, suppose that  \( (f_n) \) and \( (f_n') \) are \( o(1)\)-close and pick \( n \) large enough such that 
\[
 P\left(f'_n(Y_0) \not = f_n(Y_0)) \mid {\textstyle \frac{1}{n}} \cdot Y_0 \in I \right) < 2^{-(k+2)}  
\]
for all intervals \( I \subset C^1 \) of length \( \ell = \min_i I_i^k \). Using~\eqref{eq: mean on interval}, it follows that  
\[
P\left[f_n'(Y_0) = 1 \mid {\textstyle \frac{1}{n}} \cdot  Y_0 \in I_i^k\right] \in \left(p_i^k - 2^{-(k+1)} - 2^{-(k+2)}, p_i^k + 2^{-(k+1)} + 2^{-(k+2)}\right)
\]
and hence, as \( p_i^k \in \{ 2^{-k}, 1-2^{-k} \} \), we have that
\[
P\left[f_n'(Y_0) = 1 \mid {\textstyle \frac{1}{n}} \cdot  Y_0 \in I_i^k\right] \in \left(2^{-(k+2)}, 1-2^{-(k+2)}\right).
\]
Repeating the argument we used to show that \( (f_n) \) was volatile, it follows that \( (f'_n) \) must also be volatile.
\end{proof}

Putting the claims together, Proposition~\ref{proposition: analogue on circle} follows.
\end{proof}

We will now continue to the proof of Theorem~\ref{theorem: noise stable and volatile}. Before we do this however, we will state and prove the following lemma.

\begin{lemma}\label{lemma: comparison}
Let \(p_n = 0.5 \), \( Y_0 \sim \unif(\mathbb{Z}_{2n})\) and \( X_0 \sim \unif(\{-1,1\}^{4n^2})  \). Further, let \( \varphi \) and \( \Phi \) be the probability density function and the cumulative density function of a standard normal distribution. Then for all large enough \( n \), and any function  \( h_{2n} \colon \mathbb{Z}_{2n} \to \{ -1 ,1 \} \), we have that
\begin{equation}
\begin{split}
&2(1 - \Phi(2)) \cdot P\left[h_{2n}(Y_0) = 1\right]
\\&\qquad \leq P\left[h_{2n}(\|X_0\| \mod 2n) = 1\right]
\\&\qquad \leq  (1 + 4\varphi(0)) \cdot  P\left[h_{2n}(Y_0) = 1\right] .
\end{split}
\end{equation}
\end{lemma}

\begin{proof}

To simplify notation, define \( \rho \) by
\[
\rho \coloneqq  P\left[h_{2n}(Y_0) = 1\right] .
\]
Then 
\begin{align*}
&   P\left[h_{2n}(\|X_0\| \mod 2n) = 1 \right] 
\\&\qquad  \leq  P\left[ \|X_0\| \mod 2n \in (0,\rho \cdot 2n) \mid \left\| X_0 \right\| \geq 2n^2 \right]
\end{align*}
and
\begin{align*}
&   P\left[ h_{2n}(\|X_0\| \mod 2n) = 1\right]
\\&\qquad \geq P\left[\|X_0\| \mod 2n \in ((1-\rho) \cdot 2n,2n) \mid \| X_0 \| \geq 2n^2 \right] 
\end{align*}
and hence it suffices to give upper and lower bounds for the  last expressions in the inequalities above respectively. 

For the lower of these bounds, we  have that
\begin{equation*}
\begin{split}
&P\left[ \|X_0\| \mod 2n \in ((1-\rho) \cdot 2n,2n) \mid \| X_0 \| \geq 2n^2 \right]
\\&\qquad =
2\sum_{k = 1}^\infty 
P\left[ \|X_0\|   \in (2n^2 + k \cdot 2n -  \rho \cdot 2n ,2n^2 +  k\cdot 2n )   \right]
\\&\qquad \to
2\sum_{k = 1}^\infty 
 \left(  \Phi(2k - 2\rho) - \Phi(2k)  \right)
 >
2\sum_{k = 1}^\infty 
  2 \rho \cdot \varphi(2k) 
\\&\qquad =
\rho \cdot 4\sum_{k = 1}^\infty 
   \varphi(2k) 
 >
\rho  \cdot 2(1 - \Phi(2)).
\end{split}
\end{equation*}

Similarly, for the upper bound we have that
\begin{equation*}
\begin{split}
& P\left[ \|X_0\| \mod 2n \in (0,\rho \cdot 2n) \mid \| X_0 \| \geq 2n^2 \right]
\\&\qquad =
2\sum_{k = 0}^\infty 
P\left[ \|X_0\|   \in (2n^2 + k \cdot 2n  ,2n^2 +  k\cdot 2n + \rho \cdot 2n )   \right]  
\\&\qquad \to
2\sum_{k = 0}^\infty 
 \left(  \Phi(2k + 2\rho) - \Phi(2k) \right)
 <
\rho \cdot 2 \sum_{k = 0}^\infty 
    2 \rho \cdot \varphi(2k  )  
\\&\qquad =
\rho \cdot 4\sum_{k = 0}^\infty 
   \varphi(2k) 
 <
\rho \cdot (1 + 4\varphi(0)).
\end{split}
\end{equation*}

From this the desired conclusion follows.

\end{proof}

We will now give a proof of Theorem~\ref{theorem: noise stable and volatile}, which will more or less be a direct translation of the previous proof to our original setting on the hypercube.

\begin{proof}[Proof of Theorem~\ref{theorem: noise stable and volatile}].
With \( (f_n) \) as in the proof of Proposition~\ref{proposition: analogue on circle}, define \( g_n \colon \{ -1,1 \}^{4n^2} \to \{ -1,1 \} \) for \( x \in \{ -1,1 \}^{4n^2} \) by
\[
g_n(x) \coloneqq f_{2n}\left(\| x \|  \mod 2n\right).
\]
The main idea of this proof is that as with the scaling above, a random walk on the level sets of \(  \{ -1,1 \}^{4n^2} \) and a simple random walk on \( \mathbb{Z}_{2n} \) move around about as much, hence  all the properties we verified for \( (f_n) \) in the proof of Proposition~\ref{proposition: analogue on circle} should transfer to \( (g_n) \). 

We now do this formally. 

\begin{claim}
\( (g_n) \) is nondegenerate.
\end{claim}

\begin{proof}
By combining Lemma~\ref{lemma: comparison} with Claim~\ref{claim: nondegeneracy} it follows directly that \( (g_n) \) is nondegenerate. 
\end{proof}

Now recall that by Lemma~\ref{lemma: OU process},
\begin{equation}
\left(\frac{ \| X_t^{(n)} \| - 2n^2}{\sqrt{2}n} \right)_t \Rightarrow \left( Z_t \right)_t
\end{equation}
where \(\left(  Z_t \right)_t \) is an Ornstein Uhlenbeck process.
Note also that by the same lemma, for all \( \varepsilon > 0 \) we have that
\begin{equation}\label{eq: mean for OU}
\E [ Z_\varepsilon \mid Z_0 = z] = ze^{-\varepsilon}
\end{equation}
and that
\begin{equation}\label{eq: var for OU}
\Var [Z_\varepsilon \mid Z_0 = z]= \frac{1-e^{-2\varepsilon}}{2}.
\end{equation}

\begin{claim}
\( (g_n) \) is noise stable.
\end{claim}

\begin{proof}
With the notation of Claim~\ref{claim: noise stability}, for all \( \delta \) and \( \ell \)  we have that
\begin{equation*}
\begin{split}
&\lim_{\varepsilon \to 0} \limsup_{n \to \infty} P\left[\left| {\textstyle \frac{1}{2n}}   \| X_\varepsilon \| - {\textstyle \frac{1}{2n}}  \| X_0\| \right| \mod 1 > \delta \ell \right]
\\&\qquad=
\lim_{\varepsilon \to 0}  P\left[\left|   Z_\varepsilon - Z_0  \right| \mod \sqrt{2} > \sqrt{2} \delta \ell \right]
\\&\qquad\leq
\lim_{\varepsilon \to 0}  P\left[\left|   Z_\varepsilon - Z_0 \right| > \sqrt{2} \delta \ell \right] 
\\&\qquad\leq 
\lim_{\varepsilon \to 0} \frac{  \E \left[ (  Z_\varepsilon - Z_0 )^2   \right] }{2\delta^2 \ell^2}
\\&\qquad= 
\lim_{\varepsilon \to 0} \frac{  
(1-e^{-2\varepsilon})/2   +  (e^{-\varepsilon} - 1 )^2     }{2\delta^2 \ell^2}
\\&\qquad= 0
\end{split}
\end{equation*}
where the second inequality follows by applying Markov's inequality and the next to last equality follows by using~\eqref{eq: mean for OU}~and~\eqref{eq: var for OU}.

From this it follows that for all \( k \),
\begin{equation}\label{eq: tuple containment}
\lim_{\varepsilon \to 0} \liminf_{n \to \infty}  P \left(\exists i \colon \left\{ {\textstyle \frac{1}{2n}}  (\| X_0\| \mod 2n) , {\textstyle \frac{1}{2n}}  (\| X_{\varepsilon }\| \mod 2n) \right\} \in I_i^k \right) = 1.
\end{equation}
Combining~\eqref{eq: mean on interval} and Lemma~\ref{lemma: comparison} (for each \( i \) and \( k \) we apply this lemma to either \( f \) or \( -f \) depending on \( p_i^k \)), we get that
\begin{equation}\label{eq: new mean on interval}
\begin{split}
&P \left[f(\| X_0\| \mod 2n) = 1 \mid {\textstyle \frac{1}{2n}}  (\| X_0\| \mod 2n) \in I_i^k \right]
\\&\qquad \in  \left( 2(1 - \Phi(2)) \cdot  2^{-(k+1)}, (1 + 4\varphi(0)) \cdot 3\cdot  2^{-(k+1)} \right) 
\\& \qquad \qquad\cup  \left( 1-(1 + 4\varphi(0)) \cdot 3\cdot 2^{-(k+1)}, 1 - 2(1 - \Phi(2)) \cdot  2^{-(k+1)} \right).
\end{split}
\end{equation}
and consequently, using~\eqref{eq: tuple containment}, we obtain
\[
\lim_{\varepsilon \to 0} \limsup_{n \to \infty}  P \left[f_n(Y_0) \not = f_n(Y_{\varepsilon}) \right] \leq 2 \cdot   (1 + 4\varphi(0)) \cdot 3\cdot  2^{-(k+1)}
\]
 As \( k \) was arbitrary, we can conclude that \( (f_n) \) is noise stable.

\end{proof}

\begin{claim}
\( (g_n) \) is volatile.
\end{claim}

\begin{proof}
Using the same argument as in the proof of Claim~\ref{claim: volatility}, for any \( k \in \mathbb{N} \) we have that
\begin{equation*}
\begin{split}
&\limsup_{n \to \infty} P \left[ \tau_{\partial g_n} > \delta \right]
\\&\qquad\leq \limsup_{n \to \infty} P\left[\left| {\textstyle \frac{1}{2n}}  \| X_0\|  -  {\textstyle \frac{1}{2n}}   \| X_\delta\|  \right| \mod 1 < 2 \cdot 2^{-k} \right]
\\
&\qquad=P\left[\left| Z_0 -  Z_\delta \right| \mod \sqrt{2} < 2 \sqrt{2} \cdot 2^{-k} \right].
\end{split}
\end{equation*}
As \( k \) was arbitrary it follows that \( (g_n) \) is volatile.
\end{proof}

\begin{claim}
\( (g_n) \) is not \( o(1) \)-close to any nonvolatile sequence \( g_n' \colon \mathbb{Z}_{2n} \to \{ -1,1 \} \).
\end{claim}

\begin{proof}
Let \( (g_n) \) and \( (g_n') \) be \( o(1) \)-close. For \( \ell \in [4n^2] \), define  \(L_\ell \coloneqq \{ x \in \{ -1,1\}^{4n^2} \colon \| x \| = \ell \} \). 
Furthermore, for \( \ell, \ell' \in [4n^2] \),   let \( \mathcal{E}_{\ell,\ell',\delta}\) be the event that \( \|  X_t  \| \) hits the level sets \( L_\ell \) and \( L_{\ell'} \) before time \( \delta \). 
Note that if for some \( \ell, \ell' \in [4n^2] \) and some \( \varepsilon > 0 \) we have that 
\begin{equation}\label{eq: levelset condition}
\begin{cases}
P\left[g_n'(X_0) = 1 \mid X_0 \in L_\ell\right] > 1 - \varepsilon \cr
P\left[g_n'(X_0) = 1 \mid X_0 \in L_{\ell' }\right] < \varepsilon
\end{cases}
\end{equation}
then
\begin{equation*}\label{eq: conclusion}
P\left[\tau_{\partial g_n'} < \delta \mid \mathcal{E}_{\ell, \ell', \delta}\right] > 1 - 2\varepsilon.
\end{equation*}

As \( (g_n) \) and \( (g_n') \) are \( o(1) \)-close and \( (g_n) \) are constant on level sets, for all \( \varepsilon > 0 \) we have that
\begin{equation}\label{eq: o(1) fcn}
\lim_{n \to \infty} P_{\ell \sim \bin(4n^2, 0.5)} \left[P\left[g'_n(X_0) = 1 \mid \| X_0 \| = \ell \right]\in [0,\varepsilon) \cup (1-\varepsilon,1 ]\right] = 1.
\end{equation}

For any \( k \geq 2 \), let et \( \mathcal{F}_{k, i, \delta} \) be the event that \( \|  X_t \|/2n \mod 1 \) has hit every level set of the interval \( I_i^k \), \( i \in [2 \cdot 3^{k-1}] \) before time \( \delta \).
Using the same argument as in the proof of Claim~\ref{claim: volatility}, we get that
\begin{equation*} 
	P\left[\cup_i\mathcal{F}_{k,i, \delta } \right]
	\geq  P \left[ \left| {\textstyle \frac{1}{2n}}   \| X_0\|  - {\textstyle \frac{1}{2n}}  \| X_\delta\| \right| \mod 1 > 2 \cdot 2^{-(k-1)} \right].
\end{equation*}
and hence
\begin{equation}\label{eq: last eq 2}
\lim_{k \to \infty} \liminf_{n \to \infty}
	P\left[\cup_i\mathcal{F}_{k,i, \delta } \right] = 1.
\end{equation}

Now by~\eqref{eq: o(1) fcn},
\begin{align*}
\lim_{n \to \infty} P \left[\exists \ell, \ell' \colon  \textnormal{\eqref{eq: levelset condition} holds and } \mathcal{E}_{\ell, \ell', \delta} \textnormal{ occur } \mid   \cup_i\mathcal{F}_{k,i, \delta }  \right] = 1
\end{align*}
and hence 
\begin{align*}
\liminf_{n \to \infty} P \left[ \tau_{\partial g_n'} < \delta \mid   \cup_i\mathcal{F}_{k,i, \delta }  \right] > 1 - 2\varepsilon.
\end{align*}
As \( k \) and \( \varepsilon \) were arbitrary, \eqref{eq: last eq 2} implies that 
\begin{align*}
\lim_{n \to \infty} P \left[ \tau_{\partial g_n'} < \delta   \right] = 1,
\end{align*}
which is the desired conclusion.

\end{proof}
The claims now together imply the conclusion of the theorem.
\end{proof}

\section*{Acknowledgements}
The author is grateful to Anders Martinsson for suggesting the sequence of functions used in the proof of Proposition~\ref{proposition: block fcn result} as an example of a noise stable, monotone and volatile sequence of Boolean functions.
The author would also like to thank her supervisor Jeffrey Steif for many  inspiring conversations and for him carefully reading through the paper, which greatly improved the quality of this manuscript.
Finally, the  author would like to thank the anonymous reviewer for many helpful comments.

\end{document}